\documentclass[12pt,a4paper]{amsart}
\usepackage{tikz, float}
\usepackage{comment}
\usepackage{caption}
\usepackage{amsmath,amsthm,amsfonts,graphicx,amssymb,amscd,dsfont,euscript,enumerate,verbatim,calc,mathtools,listings}
\usepackage[hidelinks]{hyperref}
\usepackage[nameinlink, noabbrev]{cleveref}
\usepackage{bookmark}
\newtheorem{thm}{Theorem}[section]

\newtheorem{cor}[thm]{Corollary}
\newtheorem{lem}[thm]{Lemma}

\newtheorem{prop}[thm]{Proposition}

\theoremstyle{definition}
\newtheorem{defn}[thm]{Definition}
\newtheorem{example}{Example}
\newtheorem*{notation}{Notation and Conventions}

\theoremstyle{remark}
\newtheorem{rmk}{Remark}

\newcommand{\pslnc}{\mathrm{PSL}_n(\mathbb{C})}

\newcommand{\psltc}{\mathrm{PSL}_2(\mathbb{C})}
\newcommand{\psltr}{\mathrm{PSL}_2(\mathbb{R})}
\newcommand{\puto}{\mathrm{PU}(2,1)}
\newcommand{\poto}{\mathrm{PO}(2,1)}
\newcommand{\suto}{\mathrm{SU}(2,1)}

\newcommand{\pslc}{\mathrm{PSL}_2(\mathbb{C})}

\newcommand{\HtC}{{\mathbf H}^{2}_{\mathbb C}}

\newtheoremstyle{named}%
{}{}{\itshape}{}{\bfseries}{.}{.5em}{\thmnote{#3}}
\theoremstyle{named}

\allowdisplaybreaks[1]
\numberwithin{equation}{section}
\begin{document}
\title{On the domination of surface-group representations in $\mathrm{PU}(2,1)$}
\author{Pabitra Barman}
\author{Krishnendu Gongopadhyay}
\address{Department of Mathematical Sciences, Indian Institute of Science Education and Research (IISER) Mohali, Knowledge City, Sector 81, S.A.S. Nagar (Mohali) 140306, Punjab, India}
\email{pabitrab@iisermohali.ac.in, pabitrabarman560@gmail.com}
\address{Department of Mathematical Sciences, Indian Institute of Science Education and Research (IISER) Mohali, Knowledge City, Sector 81, S.A.S. Nagar (Mohali) 140306, Punjab, India}
\email{krishnendu@iisermohali.ac.in}

\begin{abstract}
    This article explores surface-group representations into the complex hyperbolic isometry group $\mathrm{PU}(2,1)$ and presents domination results for a special class of representations called the $T$-bent representations. Let $S_{g,k}$ be a surface of genus $g$ with $k > 0$ punctures and negative Euler characteristic. We prove that for a $T$-bent representation $\rho: \pi_1(S_{g,k}) \rightarrow \mathrm{PU}(2,1)$, there exists a discrete and faithful representation $\rho_0: \pi_1(S_{g,k}) \rightarrow \mathrm{PO}(2,1)$ that dominates $\rho$ in the Bergman translation length spectrum, while preserving the lengths of the peripheral loops.
\end{abstract}
\maketitle
\section{Introduction}
Let $S_{g,k}$ be an oriented surface of negative Euler characteristic with genus $g$ and puncture $k \geq 0$. When there is no puncture, that is, $k=0$, we simply denote such a surface by $S_g$. In all other cases, we assume $k \geq 1$. 

The study of surface-group representations $\rho: \pi_1(S_{g,k}) \rightarrow G$ into different Lie groups $G$ (preferably of higher rank) is a significant aspect in Higher Teichm\"{u}ller theory (see \cite{Wienhard} for a survey). In this article, we focus on representations into $\puto$. Recall that, $\puto$ is the group of holomorphic isometries of the two dimensional complex hyperbolic plane $\HtC$. Despite sharing some structural similarities with their real counterpart, surface-group representations into $\puto$ exhibit distinct geometric and dynamical behaviors that are still not fully understood. Our aim is to explore these representations in relation to domination, a notion that has been used in analyzing surface-group representations in certain other settings.

Classically, the concept of domination is given as follows.  Let $\rho_1, \rho_2 :\pi_1(S_{g,k}) \rightarrow \pslc$ be two representations. We say $\rho_2$ \emph{dominates} $\rho_1$ if there exists $\lambda\le 1$ such that \[ \ell_{\rho_1}(\gamma) \le \lambda \cdot \ell_{\rho_2}(\gamma) \] for all $\gamma \in \pi_1(S_{g,k})$, where $\ell_{\rho_i}(\gamma)$ denotes the translation length of $\rho_i(\gamma)$ in $\mathbb{H}^3$. The domination is said to be \emph{strict} if $\lambda<1$. 

Domination in the context of surface-group representations first arose in the study of anti-de Sitter (AdS) $3$-manifolds, cf. \cite{kp}. It is an important geometric phenomenon with significant implications for surface-group representations. In 2015, Gu\'eritaud-Kassel-Wolff \cite{GKW} showed that for a closed surface $S_g$ and a non-Fuchsian representation $\rho : \pi_1(S_g) \rightarrow \psltr$, there exists a Fuchsian representation $j: \pi_1(S_g) \rightarrow \psltr$ that strictly dominates $\rho$. This theorem has direct application in the construction of compact AdS 3-manifolds. Around the same period, Deroin-Tholozan \cite{DT} showed that a representation $\rho:\pi_1(S_g) \rightarrow \psltc$ can be strictly dominated by a Fuchsian representation $\rho_0:\pi_1(S_g) \rightarrow \psltr$, unless $\rho$ is itself Fuchsian. Other progresses in dominating surface-group representations can also be found in \cite{Tholozan}, \cite{GGKW}, \cite{DMSV}, \cite{Sagman}.
 
For a punctured surface $S_{g,k}$, Gupta-Su \cite{GSu} showed that for a non-Fuchsian representation $\rho:\pi_1(S_{g,k}) \rightarrow \psltc$ with $n\ge1$, there exists a Fuchsian representation $\rho_0:\pi_1(S_{g,k}) \rightarrow \psltr$ that strictly dominates $\rho$, preserving the lengths of the peripheral loops. Note that, if we allow the lengths of the peripheral curves to increase also, then such a domination can be easily constructed by the `strip deformation', as discussed in \cite{GSu}. The above result has been generalized in \cite{BGup} for the Hilbert length spectrum and the translation length spectrum to the case when the target-group is $\pslnc$.

While this notion of domination has been studied in some depth for target groups like  $\pslnc$, $n \geq 2$, much less is known when the target is $\puto$. We shall explore this in this paper. For this case, we turn to a specific kind of representations, namely, the \emph{$T$-bent representations}, introduced by Will \cite{WillBending}. These representations are constructed using a fixed ideal triangulation of the surface and a certain boundary map.

Let $T$ be an ideal triangulation on a punctured surface $S_{g,k}$ and $\mathcal{F}_{\infty}$ be the Farey set - the set of points on $\partial \widetilde{S}_{g,k}$ corresponding to the lifts of the punctures. 
\begin{defn}  \cite{WillBending}   ($T$-bent representation)
    A representation $\rho: \pi_1(S_{g,k}) \rightarrow $Isom($\HtC$) is called \emph{$T$-bent} if there exists a $\rho$-equivariant map (known as \emph{framing}) $\phi :\mathcal{F}_{\infty} \rightarrow \partial \HtC$ such that for any ideal triangle $\Delta \in \widetilde{T}$ with vertices $a,b,c \in \mathcal{F}_{\infty}$, the images $\phi(a), \phi(b), \phi(c)$ form a real ideal triangle in $\HtC$.
\end{defn}

An element $g\in$ Isom($\HtC$) acts on a pair $(\rho,\phi)$ by $g\cdot (\rho,\phi)=(g\rho g^{-1}, g\circ \phi)$. Following the notation of Will, let $\mathcal{BR}_T$ denote the Isom($\HtC$)-classes of the pair $(\rho,\phi)$, i.e., \[ \mathcal{BR}_T:= \{ (\rho,\phi) \}/\text{Isom}(\HtC). \]
The surface-group representation variety $\text{Rep}_{(\pi_1(S_{g,k}),\puto)}$ has real dimension $16g-16+8k$. However, Will showed that $\mathcal{BR}_T$ has real dimension $12g-12+6k$ (see \Cref{willbijection}).

In this article, we show domination for these $T$-bent representations into $\puto$. Let $\sigma(A)$ denote the spectral radius of $A$. Recall that \[ \sigma(A) := \max \{|\lambda| : \lambda \text{ is an eigenvalue of } A \}. \] Let $\ell(A)$ denote the translation length of an element $A \in \suto$ with respect to the Bergman metric, a Riemannian metric on $\HtC$ (see \cref{bergmet}). From Proposition 3.10 in \cite{ParkerTrace}, we know that $\ell(A)= 2\ln\sigma(A).$ For an element $A\in \puto$, we take a representative $\tilde{A}\in \mathrm{U}(2,1)$ and define \[\ell(A):=2\ln \sigma(\tilde{A}). \] Note that any two lifts of $A$ to $\mathrm{U}(2,1)$ differ by multiplication with a scalar $\lambda$ such that $|\lambda| = 1$. So they have the same spectral radius. Consequently, Bergman translation length is well-defined on $\puto$.
\begin{defn}\label{domi}(Domination for $\puto$-representations)
    Given two representations $\rho_1, \rho_2:\pi_1(S_{g,k}) \rightarrow \puto,\ \rho_2$ is said to \emph{dominate} $\rho_1$ in the Bergman translation length spectrum if \[ \ell(\rho_1(\gamma)) \le \ell(\rho_2(\gamma)) \] for all $\gamma \in \pi_1(S_{g,k})$. The domination is said to be strict if $\ell(\rho_1(\gamma)) < \ell(\rho_2(\gamma))$ for all $\gamma \in \pi_1(S_{g,k})$.
\end{defn}
Let $\poto$ be the identity component of the isometry group of ${\bf H}^2_{\mathbb R}$. Our main result is the following. 
\begin{thm}\label{mainthm1}
    Let $S_{g,k}$ be an oriented surface of negative Euler characteristic with at least one puncture.  For a $T$-bent representation $\rho:\pi_1(S_{g,k}) \rightarrow \puto$, there exists a discrete and faithful representation $\rho_0:\pi_1(S_{g,k}) \rightarrow \poto$ that dominates $\rho$. Moreover, if $\gamma \in \pi_1(S_{g,k})$ is a peripheral loop, then the Bergman translation length of $\rho(\gamma)$ remains unchanged.
\end{thm}
This is partial generalization of Deroin-Tholozan's work in \cite{DT} for the punctured surfaces.  We have used the ${\tt Z}$-invariant (see \Cref{ZinvSec}) associated with a pair of adjacent real ideal triangles to establish our result. These invariants also motivate a notion of \emph{bending fiber} for $T$-bent representations into $\puto$, analogous to the construction in \cite[Definition 2.14]{BGup} for $\mathrm{PSL}_n(\mathbb{C})$ using Fock-Goncharov coordinates. In our context, we adopt the following version:
\begin{defn}[Bending fiber for $\puto$-representations]\label{bendingfiber}
Two $T$-bent representations $\rho_1, \rho_2 : \pi_1(S_{g,k}) \to \puto$ are said to lie in the \emph{same bending fiber} if the absolute values of the ${\tt Z}$-invariants assigned to each edge of $T$ coincide. In that case, we say that $\rho_1$ and $\rho_2$ are related by a \emph{bending deformation}..
\end{defn}

With this terminology, we can restate \Cref{mainthm1} in the language of bending fibers. 
\begin{thm}\label{mainthm2}
    Let $S_{g,k}$ be an oriented surface of negative Euler characteristic with at least one puncture and $\rho:\pi_1(S_{g,k}) \rightarrow \puto$ be a $T$-bent representation. Then the unique discrete and faithful representation $\rho_0:\pi_1(S_{g,k}) \rightarrow \poto$ in the bending fiber of $\rho$ dominates $\rho$, keeping the Bergman translation lengths of the peripheral loops unchanged.  
\end{thm}
For more general surface-group representations into $\puto$, we expect to prove analogous dominating-type results in the spirit of \cite{BGup}, using the invariants associated with triples and quadruples of flags of $\HtC$, as developed in \cite{WillMarche}.

The trace of an element $A\in \suto$ also determines its geometric action and is directly associated with $\ell(A)$, as described in \cite{ParkerTrace}. As discussed earlier, any two lifts of an element in $\puto$ differ by a unitary scalar. While the trace of a lift $\tilde{A}$ may vary under such scalar multiplication, its modulus remains unchanged. Therefore, $|\mathrm{tr}(\rho(\gamma))|$ is a well-defined function on $\puto$. As a corollary of \Cref{mainthm1}, we also obtain:

\begin{cor}\label{trace_ineq}
Let $\rho$ and $\rho_0$ be as in \Cref{mainthm1}. Then \[ |\mathrm{tr}(\rho(\gamma))| \le |\mathrm{tr}(\rho_0(\gamma))| \] for all $\gamma \in \pi_1(S_{g,k})$.
\end{cor}
We have also experimentally observed that the inequality does not extend for the discriminator function $f(\mathrm{tr}(\rho(\gamma)))$ (see \Cref{discri}). Further details are discussed in \hyperref[sec:appendix]{the Appendix}.
\begin{rmk}
    Note that \Cref{mainthm1} does not guarantee the existence of a representation $\rho_0:\pi_1(S_{g,k}) \to \poto$ that strictly dominates $\rho$. For instance, if a non-trivial, non-peripheral curve $\gamma \in \pi_1(S_{g,k})$ passes through a subsurface of $S_{g,k}$ where the ${\tt Z}$-invariants are already real and positive, we get $\rho_0(\gamma)= \rho(\gamma)$.
\end{rmk}
\subsection*{\textbf{Acknowledgements}} The authors would like to thank Subhojoy Gupta for indicating this problem to us and for motivations and useful discussions. The authors also thank Fanny Kassel and John R. Parker for valuable comments and discussions. 

This research was supported in part by the International Centre for Theoretical Sciences (ICTS), Tata Institute of Fundamental Research, through participation in the program \emph{New Trends in Teichmüller Theory}. 

Barman sincerely thanks Pranab Sardar for valuable support during his post doctoral stay at IISER Mohali. Gongopadhyay acknowledges ANRF research Grant CRG/2022/003680, and DST-JSPS Grant DST/INT/JSPS/P-323/2020.

\section{Preliminaries}
This section delves into the preliminaries and foundational setup utilized in this article. 
\subsection{Complex hyperbolic plane and its isometries} Let us consider the non-degenerate, indefinite Hermitian form on $\mathbb{C}^3$ of signature (2,1) given by:
\begin{equation}\label{hform}
    \langle \textbf{z}, \textbf{w} \rangle =z_1\bar{w}_3+z_2\bar{w}_2+z_3\bar{w}_1
\end{equation}
where $\textbf{z} = (z_1, z_2, z_3)^t$ and $\textbf{w} = (w_1, w_2, w_3)^t \in \mathbb{C}^3$.
In the canonical basis of $\mathbb{C}^3$, this Hermitian form is given by the matrix 
\[ J= \begin{pmatrix}
    0&0&1\\0&1&0\\1&0&0
\end{pmatrix}, \] 
so that $\langle \textbf{z}, \textbf{w} \rangle = \textbf{z}^* J \textbf{w}$.

Let $V_-$ and $V_0$ denote the set of negative and null vectors in $\mathbb{C}^{2,1}$ respectively, i.e., 
\[ V_-= \{\textbf{v}\in \mathbb{C}^{2,1}\ |\ \langle \textbf{v}, \textbf{v} \rangle <0 \} \text{ and } V_0= \{\textbf{v}\in \mathbb{C}^{2,1}\ |\ \langle \textbf{v}, \textbf{v} \rangle =0 \}. \]
Note that both $V_-$ and $V_0$ are invariant under scalar multiplication by nonzero complex numbers, i.e., if $\textbf{v} \in V_-$ (respectively, $V_0$), then $c\textbf{v} \in V_-$ (respectively, $V_0$) for all $c \in \mathbb{C}^*$. 

The complex hyperbolic 2-space $\mathbb{H}^2_{\mathbb{C}}$ is defined as the projectivization of $V_-$:
\[ \HtC := \mathbb{P}V_-. \] 
Its boundary is given by the projectivization of null vectors: $\partial \HtC =\mathbb{P}V_0$.
This realization of $\HtC$ corresponds to the \emph{Siegel model} (see \cite[Chapter~4]{GoldmanCH}), which is particularly convenient for describing the boundary via Heisenberg coordinates.

The \emph{Heisenberg coordinates} (see \cite[\S 4]{ParkerNotes} or \cite[\S 2]{WillBending} for more details) of a point $(z_1,z_2,1)^t \in \partial \HtC$ is given by $[\zeta,t]$, where $z_2=\zeta \sqrt{2}$ and $z_1=-|\zeta|^2+it$.

The metric on $\HtC$, known as the Bergman metric is given by
\begin{equation}\label{bergmet}
    \cosh^2\bigg(\frac{d(z,w)}{2}\bigg)=\frac{\langle \textbf{z},\textbf{w}\rangle \langle \textbf{w},\textbf{z} \rangle}{\langle \textbf{z},\textbf{z}\rangle \langle \textbf{w},\textbf{w} \rangle}. 
\end{equation}

The holomorphic isometry group of $\HtC$ is the projective unitary group $\puto$, consisting of projective linear transformations that preserve the Hermitian form (up to scalar multiple). Explicitly, $\puto = \mathrm{U}(2,1)/ \{\lambda I \}$, where $\mathrm{U}(2,1)$ is the group of matrices $A \in \mathrm{GL}_3(\mathbb{C})$ satisfying $A^*JA = J$, with $J$ being the Hermitian matrix corresponding to the chosen form and $\lambda \in \mathbb{C}^*$. This group acts transitively and holomorphically on $\HtC$, preserving both the complex structure and the Bergman metric. The isometry group Isom($\HtC$) is generated by $\puto$ and the antiholomorphic reflections along real Lagrangian planes (see \cref{realsym}).

Much like in the real hyperbolic case, elements of $\puto$ are classified into three types based on their fixed point behavior: \emph{loxodromic} (fixing exactly two points on the boundary $\partial\HtC$), \emph{parabolic} (fixing exactly one point on $\partial\HtC$) and \emph{elliptic} (fixing a point in $\HtC$). This classification can be determined using the trace of a lift to $\suto$, as described below.
\begin{thm}\label{discri}\cite{GoldmanCH}
    Let $f(z) = |z|^4 - 8\Re(z^3) + 18|z|^2 - 27$. For $A \in \suto$, the following holds:
    \begin{itemize}
        \item $A$ has an eigenvalue $\lambda$ with $|\lambda| \neq 1$ if and only if $f(\mathrm{tr}(A)) > 0$; in this case, $A$ is loxodromic.
        \item $A$ has a repeated eigenvalue if and only if $f(\mathrm{tr}(A)) = 0$; in this case, $A$ is either parabolic or special elliptic (i.e. elliptic with a repeated eigenvalue).
        \item $A$ has distinct eigenvalues all of unit modulus if and only if $f(\mathrm{tr}(A)) < 0$; in this case, $A$ is elliptic.
    \end{itemize}
\end{thm}

\subsection{Totally geodesic subspaces and ideal triangles} There are two different types of totally geodesic maximal subspaces in $\HtC$ (see \cite[\S 5.2]{ParkerNotes}):
\begin{itemize}
    \item \textbf{Complex lines} are the intersections of $\HtC$ with complex projective lines whenever these intersections are non-empty. They are biholomorphic to the Poincar\'e disk, and thus form embedded copies of ${\mathbf H}^{1}_{\mathbb C}$ inside $\HtC$.
    \item \label{realsym} \textbf{Totally real Lagrangian planes} (also called real planes) are the intersections of $\HtC$ with totally real 2-planes in $\mathbb{C}^3$, whenever these intersections are non-empty. They are isometric to the Klein–Beltrami model of the real hyperbolic plane $\mathbb{H}^2_\mathbb{R}$. Each such plane determines an antiholomorphic involution of $\HtC$ fixing it pointwise; this involution is called its \emph{real symmetry}.
\end{itemize}
The ideal boundary of these two subspaces are called $\mathbb{C}$-circles and $\mathbb{R}$-circles respectively. Real planes are particularly important in this article, since we want the framing to map the vertices of each of the ideal triangles to the vertices of some real ideal triangle for a representation to be $T$-bent. In particular, a triangle in $\HtC$ is said to be a \emph{real ideal triangle} if its vertices lie on the boundary of a real plane, and hence, can be thought of as the ideal triangle in a copy of $\mathbf{H}^2_\mathbb{R} \subset \HtC$.

Unlike in $\mathbf{H}^2_\mathbb{R}$, ideal triangles are not uniquely determined up to isometry in $\HtC$. Instead, they are distinguished by an invariant, called Cartan invariant:
\begin{defn}[Cartan invariant] The \emph{Cartan invariant} of an ideal triangle $(v_1,v_2,v_3)$ is defined to be \[ \mathbb{A}(v_1,v_2,v_3):=arg (-\langle \textbf{v}_1, \textbf{v}_2 \rangle \langle \textbf{v}_2, \textbf{v}_3 \rangle \langle \textbf{v}_3, \textbf{v}_1 \rangle ), \] where $\textbf{v}_i$ are any lifts of $v_i$ to $\mathbb{C}^3$. This quantity is independent of the choice of lifts. 
\end{defn}
The Cartan invariant characterizes the ideal triangles as follows:
\begin{prop}\cite[\S 7]{GoldmanCH}
    Let $x=(x_1,x_2,x_3)$ and $y=(y_1,y_2,y_3)$ be two ideal triangles. Then there is a holomorphic (resp. antiholomorphic) isometry $A\in$ Isom $(\HtC)$ with $A(x_i)=y_i$ if and only if they have same (resp. opposite) Cartan invariant. Moreover, an ideal triangle has Cartan invariant 0 (resp. $\pm \pi/2$) if and only if it is contained in a real plane (resp. complex line).
\end{prop}

\subsection{The $T$-bent representations and their ${\tt Z}$-invariants}\label{ZinvSec} This part mainly revisits \cite{WillBending} and is included here for the sake of completeness. We begin by recalling a geometric property about pairs of adjacent real ideal triangles. This property is the foundation for defining the ${\tt Z}$-invariant.
\begin{prop}\label{zprop}\cite[Lemma 2]{WillBending}
    Let $\tau_1$ and $\tau_2$ be two adjacent real ideal triangles. Then there exists a unique complex number $z \in \mathbb{C} \setminus \{-1,0\}$ for which there is an element of $\mathrm{PU}(2,1)$ mapping the ordered pair $(\tau_1, \tau_2)$ to the ordered pair $(\tau_0, \tau_z)$, where $\tau_0$ and $\tau_z$ are as follows:
    \[ \tau_0= \bigg(\begin{bmatrix}
           1\\0\\0
         \end{bmatrix}, \begin{bmatrix}
           -1\\-\sqrt{2}\\1
         \end{bmatrix}, \begin{bmatrix}
           0\\0\\1
         \end{bmatrix} \bigg) \ \text{ and } \   \tau_z= \bigg(\begin{bmatrix}
           1\\0\\0
         \end{bmatrix}, \begin{bmatrix}
           0\\0\\1
         \end{bmatrix}, \begin{bmatrix}
          -|z|^2\\z\sqrt{2}\\1
         \end{bmatrix} \bigg). \]
\end{prop}
\begin{rmk}
    In Heisenberg coordinates, the triangles $\tau_0$ and $\tau_z$ are given by \[ \tau_0=(\infty, [-1,0],[0,0])\ \text{ and }\ \tau_z=(\infty,[0,0],[z,0]). \]
\end{rmk}
The unique $z$ in \Cref{zprop} is defined to be the ${\tt Z}$-invariant of the ordered pair of adjacent real ideal triangles $(\tau_1, \tau_2)$ and denoted by ${\tt Z}(\tau_1, \tau_2)$. It can also be explicitly defined as follows:
\begin{defn}[${\tt Z}$-invariant]
    Let $\tau_1=(p_1,p_2,p_3)$ and $\tau_2=(p_3,p_4,p_1)$ be two adjacent real ideal triangles. Let $C_{13}$ be the unique complex line spanned by $p_1$ and $p_3$ with a polar vector $\textbf{v}$. Then the ${\tt Z}$-invariant of $(\tau_1,\tau_2)$ is defined to be \[ {\tt Z}(\tau_1, \tau_2):= -\frac{\langle \textbf{p}_4,\textbf{v} \rangle\langle \textbf{p}_2, \textbf{p}_1 \rangle }{ \langle \textbf{p}_2,\textbf{v} \rangle\langle \textbf{p}_4, \textbf{p}_1 \rangle} \] for any lift $\textbf{p}_i$ of $p_i$. Equivalently, \[ {\tt Z}(\tau_1, \tau_2):= -\frac{\langle \textbf{p}_4, \textbf{p}_1 \boxtimes \textbf{p}_3 \rangle\langle \textbf{p}_2, \textbf{p}_1 \rangle }{ \langle \textbf{p}_2, \textbf{p}_1 \boxtimes \textbf{p}_3 \rangle\langle \textbf{p}_4, \textbf{p}_1 \rangle} .\]
\end{defn}
\begin{rmk}
    It is straightforward to verify that ${\tt Z}(\tau_1, \tau_2)$ is invariant under the action of $\mathrm{PU}(2,1)$ and ${\tt Z}(\tau_1, \tau_2)=\overline{{\tt Z}(\tau_2, \tau_1)}$.
\end{rmk}
\begin{rmk}
    For the ordered pair $(\tau_0, \tau_z)$ in \Cref{zprop}, choosing the polar vector to be $\textbf{v} = (0,1,0)^t$ yields ${\tt Z}(\tau_0, \tau_z) = z$. This confirms the consistency of the explicit and geometric definitions of the ${\tt Z}$-invariant.
\end{rmk}
The following proposition plays a key role in the construction of the inverse map $\psi^{-1}$ in \Cref{willbijection}, as it guarantees the existence of a specific isometry in $\mathrm{PU}(2,1)$ associated with the ${\tt Z}$-invariant:
\begin{prop}\label{mzprop}\cite[Proposition 6]{WillBending}
    Let $\tau_1=(p_1,p_2,p_3)$ and $\tau_2=(p_3,p_4,p_1)$ be two adjacent real ideal triangles with ${\tt Z}(\tau_1, \tau_2)=z=xe^{i \alpha} \in \mathbb{C}\setminus\{-1,0\}$. Then there is a unique real symmetry $M_z$ such that $M_z(p_1,p_3)=(p_3,p_1)$ and $M_z(p_2,p_4)=(p_4,p_2)$ (see \Cref{mzflips}).
\end{prop}
\begin{figure}[tph]
	\centering
	\begin{tikzpicture}
	\draw (-0.4,1.2)-- (1,2.5);
	\draw (-0.4,1.2)-- (1,0);
	\draw (1,2.5)-- (2.5,1.5);
	\draw (2.5,1.5)-- (1,0);
	\draw (1,2.5)-- (1,0);
	\draw (0.98,2.7) node{$p_1$};
	\draw (-.7,1.2) node{$p_2$};
	\draw (0.98,-0.25) node{$p_3$};
	\draw (2.8,1.4) node{$p_4$};
	\end{tikzpicture}
	\caption{$M_z$ \emph{flips} the vertices of the adjacent real ideal triangles.}
	\label{mzflips}
\end{figure} 
Considering $\tau_1$ and $\tau_2$ as in \Cref{zprop}, the matrix $M_z$ in the above proposition is given by \begin{equation}\label{mz}
    M_z=M_{x,\alpha}= \begin{pmatrix}
        0&0&x\\0&e^{i\alpha}&0\\1/x&0&0
    \end{pmatrix}.
\end{equation}
We can easily check that $M_z\overline{M_z}=\mathrm{Id}$.

The preceding theorem ensures that the ${\tt Z}$-invariants can indeed be used to realize a parametrization of the Isom($\HtC$)-classes of $T$-bent representations.
\begin{thm}\label{willbijection}\cite[Theorem 1]{WillBending}
    Let $T$ be an ideal triangulation on $S_{g,k}$. Then there exists a bijection \[\psi : \mathcal{BR}_T \rightarrow (\mathbb{C} \setminus \{-1,0\} )^{6g-6+3k}/\mathbb{Z}_2 \] where $\mathbb{Z}_2$ acts on $(\mathbb{C} \setminus \{-1,0\} )^{6g-6+3k}$ by conjugation.  
\end{thm}
Note that there are exactly $6g - 6 + 3k$ edges in an ideal triangulation on $S_{g,k}$. To define $\psi((\rho, \phi))$, the ${\tt Z}$-invariant is assigned to each edge of $T$, using the pair of adjacent real ideal triangles determined by the image of $\phi$. We refer to \cite{WillBending} for the precise construction and technical details. 

From a given tuple $\textbf{c}\in (\mathbb{C} \setminus \{-1,0\} )^{6g-6+3k}$, obtaining $(\rho,\phi)=\psi^{-1}(\textbf{c})$ is a standard technique, as discussed in many places, e.g., \cite[\S 5.2]{CTT}, \cite[\S 2.2]{BGup} (though in different setups), and in particular \cite[\S 4]{WillBending}. We are going to discuss it briefly here and again refer to \cite{WillBending} for the technical details. 

We first determine the framing map $\phi$. For this, we take a lift of the triangulation $T$ in the universal cover $\widetilde{S}_{g,k}$. Then we start with a pair of adjacent ideal triangles in $\widetilde{T}$. The common edge between them has a complex number assigned to it. We map the four ideal vertices involved so that the image of the vertices of each triangle lies on the boundary of some real plane in $\partial \HtC$, and the ${\tt Z}$-invariant of the resulting adjacent real ideal triangles equals the given complex number. This is possible by \Cref{zprop}. We then extend $\phi$ recursively by moving across adjacent triangles in the triangulation. This construction determines $\phi$ up to post-composition by an element of Isom($\HtC$).

To determine the representation $\rho$, we use the modified dual graph (also called the monodromy graph) $\Gamma_T$. To obtain it (see \Cref{monod}), we start with the dual graph of $T$ and then blow up each of its vertices into a small triangle entirely contained within the corresponding ideal triangle.
\begin{figure}[tph]
\centering
\begin{tikzpicture}[scale=0.6]
\draw (3,5) -- (0,0) -- (5,0) -- (3,5) -- (8,3) -- (5,0);
\draw [line width=0.5mm,   black] (3.4,2) -- (2.2,2.2) -- (2.8,1.2) -- (3.4,2) -- (4.8,2.5) --(6.1,2.3) -- (5.5,3.4) -- (4.8,2.5);
\draw [line width=0.5mm,   black] (0.8,3.2) -- (2.2,2.2);
\draw [line width=0.5mm,   black] (2.8,-0.8) -- (2.8,1.2);
\draw [line width=0.5mm,   black] (6.1,2.3) -- (7.3,1.3);
\draw [line width=0.5mm,   black] (5.5,3.4) -- (6,4.7);
\draw [line width=0.5mm,  red] plot [smooth, tension=0.4] coordinates { (0.4,3.5) (0.8,3.4) (2.2,2.4) (3.4,2.2) (4.8,2.3) (6.1,2.1) (7.1, 1.3)};
\node [red] at (0.4,3.9) {$\gamma$};
\node at (5,-0.4) {$T$};
\node at (6.4,4.74) {$\Gamma_T$};
\node at (3.3,1.5) {$t$};
\node at (2.5,0.5) {$e$};
\end{tikzpicture}
\setlength{\belowcaptionskip}{-8pt}
\caption{Part of the modified dual graph (bold).}
\label{monod}
\end{figure}

Notice that, there are two kinds of edges on $\Gamma_T$:
\begin{itemize}
    \item \emph{$e$-edges}: these are the edges that are transverse to $T$.
    \item \emph{$t$-edges}: these are the edges that lie entirely within an ideal triangle of $T$.
\end{itemize}
We consider the oriented lift $\Gamma_{\widetilde{T}}$ where the orientation is induced by the orientation of the surface. Then we assign the matrix $M(e)=M_z$ (see \Cref{mzprop} and \cref{mz}) with each of the oriented $e$-edges, where $z$ is the ${\tt Z}$-invariant of the two ordered adjacent real ideal triangles coming from the images of $\phi$. With each of the oriented $t$-edges, we assign the matrix \begin{equation}\label{epsi}
    M(t)=\mathcal{E}= \begin{pmatrix}
    -1 & \sqrt{2}& 1 \\ -\sqrt{2} & 1 & 0\\ 1&0&0
\end{pmatrix}\in \puto.
\end{equation} For the opposite orientation of a $t$-edge, we assign $\mathcal{E}^{-1}$. This is an elliptic element of order $3$ that cyclically permutes the 3 vertices of the real ideal triangle $\tau_0$ in \Cref{zprop}. Let $\gamma \in \pi_1(S_{g,k})$. We fix a lift of the base point. Then $\widetilde{\gamma}$ can be uniquely homotoped to the edges of $\Gamma_{\widetilde{T}}$ in such a way that the $e$-edges and the $t$-edges appear alternatively, beginning with a $t$-edge (see \Cref{monComp}).
\begin{figure}[tph]
	\centering
	\begin{tikzpicture}
	\draw (-0.4,1.2)-- (1,2.5);
	\draw (-0.4,1.2)-- (1,0);
	\draw (1,2.5)-- (2.5,1.5);
	\draw (2.5,1.5)-- (1,0);
	\draw (4.6,1.2)-- (6,2.5);
	\draw (6,2.5)-- (6,0);
	\draw (4.6,1.2)-- (6,0);
	\draw (1,2.5)-- (1,0);
	\draw (0.93,2.83) node[anchor=north west] {$a$};
	\draw (-.8,1.43) node[anchor=north west] {$b$};
	\draw (0.93,0.14) node[anchor=north west] {$c$};
	\draw (5.6,2.91) node[anchor=north west] {$\gamma \cdot a$};
	\draw (3.7,1.48) node[anchor=north west] {$\gamma \cdot b$};
	\draw (5.6,0.14) node[anchor=north west] {$\gamma \cdot c$};
	\draw (2.52,1.75) node[anchor=north west] {$d$};
	\draw (3.26,1.66) node[anchor=north west] {$\cdots$};
	\draw [line width=1.4pt] (0.02,1.98)-- (0.39,1.49);
	\draw [line width=1.4pt] (0.39,1.49)-- (0.42,0.95);
	\draw [line width=1.4pt] (0.42,0.95)-- (0.75,1.22);
	\draw [line width=1.4pt] (0.75,1.22)-- (0.39,1.49);
	\draw [line width=1.4pt] (0.42,0.95)-- (0.09,0.45);
	\draw [line width=1.4pt] (0.75,1.22)-- (1.32,1.24);
	\draw [line width=1.4pt] (1.32,1.24)-- (1.67,1.69);
	\draw [line width=1.4pt] (1.67,1.69)-- (1.77,1.16);
	\draw [line width=1.4pt] (1.77,1.16)-- (1.32,1.24);
	\draw [line width=1.4pt] (1.77,1.16)-- (2.18,0.78);
	\draw [line width=1.4pt] (5.02,1.98)-- (5.39,1.49);
	\draw [line width=1.4pt] (5.75,1.22)-- (5.39,1.49);
	\draw [line width=1.4pt] (5.39,1.49)-- (5.42,0.95);
	\draw [line width=1.4pt] (5.42,0.95)-- (5.75,1.22);
	\draw [line width=1.4pt] (5.42,0.95)-- (5.09,0.45);
	\draw [line width=1.4pt] (5.75,1.22)-- (6.32,1.24);
	\draw [line width=1.4pt] (1.67,1.69)-- (1.93,2.23);
	\draw (0,1.41) node[anchor=north west] {$t$};
	\draw (.9,1.56) node[anchor=north west] {$e$};
	\draw [line width=1.1pt,  red] plot [smooth, tension=.6] coordinates { (0.52, 0.78) (0.78, 0.99) (1.29, 1.03) (1.77, 0.85) (2.12, 0.65) (2.96, 0.29) (4.17, 0.16) (4.86, 0.41) (5.3, 0.9)};
	\draw [color=red](3.45,0.2) node[anchor=north west] {$\widetilde\gamma$};
	\end{tikzpicture}
	\caption[Computing monodromy of a curve]{Computing $\rho(\gamma)$.}
	\label{monComp}
\end{figure} 
Let $\widetilde{\gamma} \sim t_1*e_1*\cdots *t_r*e_r$. If certain $t_i$ are oriented in the direction opposite to $\widetilde{\gamma}$, we have $M(t_i)=\mathcal{E}^{-1}$. Then $\rho(\gamma)$ belongs to the conjugacy class of $ M(e_r)M(t_r)\cdots M(e_1)M(t_1)=M_{z_r}\mathcal{E}^{\delta_r}\cdots M_{z_1}\mathcal{E}^{\delta_1}$ where $\delta_i=\pm1$ (see \cite[Proposition 9]{WillBending}).
\begin{rmk}
    Note that this process determines $\rho(\gamma)$ only up to conjugacy. Since eigenvalues are invariant under conjugation, determining $\rho(\gamma)$ up to conjugacy is sufficient for our proof. With a slight abuse of notation, we shall therefore write
    \[
        \rho(\gamma) =M_{z_r}\mathcal{E}^{\delta_r}\cdots M_{z_1}\mathcal{E}^{\delta_1}.
    \]
\end{rmk}

A triangulation on a surface $S_{g,k}$ is called \emph{bipartite} if its dual graph is bipartite. From now on, we shall only consider the bipartite ideal triangulations on the punctured surface $S_{g,k}$. This is necessary because, by \cite[Proposition 10]{WillBending}, the associated isometries lie in $\puto$ if and only if the triangulation is bipartite. The following result ensures that such triangulations always exist:
\begin{prop}\cite[Proposition 11]{WillBending}
    Every oriented punctured surface of negative Euler characteristic admits a bipartite ideal triangulation.
\end{prop}

\section{Proof of the theorem}
\begin{notation}
Throughout this paper, we view a representation 
$\rho: \pi_1(S_{g,k}) \to \mathrm{PU}(2,1)$ as a map into the projective unitary group. Technically, for each $\gamma$, $\rho(\gamma)$ is an equivalence class of matrices in $\mathrm{U}(2,1)$ modulo multiplication by scalars of modulus one. However, whenever convenient we choose a representative $\tilde\rho(\gamma) \in \mathrm{U}(2,1)$ and write expressions such as
\[
\rho(\gamma) = M_{z_r}\mathcal{E}^{\delta_r}\cdots M_{z_1}\mathcal{E}^{\delta_1},
\]
meaning that $\tilde\rho(\gamma) = M_{z_r}\mathcal{E}^{\delta_r}\cdots M_{z_1}\mathcal{E}^{\delta_1}$ for some choice of lift. All quantities we use (e.g.\ spectral radius $\sigma$, Bergman translation length $\ell = 2\ln\sigma$, and $|\mathrm{tr}|$) are invariant under multiplication by a unit scalar, so this convention does not affect any statement or proof.
\end{notation}

We now proceed to the proof of the main theorem. Let $\rho:\pi_1(S_{g,k}) \rightarrow \puto$ be a given generic representation, and $\gamma \in \pi_1(S_{g,k})$. Let us fix an ideal triangulation $T$ on $S_{g,k}$. Let \[\mathcal{Z}=\{z_j\}_{j=1}^{3(2g+k-2)}\] be the set of all edge invariants associated with the edges of $T$. Let $z_j=x_je^{i\alpha_j}$, where $\alpha_j\in [0,2\pi)$ for all $j$. Let us denote 
\[ \mathcal{X}=\{x_j\}_{j=1}^{3(2g+k-2)} \subset \mathbb{R}^+ \text{ and } \Theta = \{\alpha_j\}_{j=1}^{3(2g+k-2)}. \] 
We know that 
\[ \rho(\gamma) = M_{z_r}\mathcal{E}^{\delta_r}\cdots M_{z_1}\mathcal{E}^{\delta_1}  \] 
where $z_i$ are the edge-invariants, $\delta_i \in \{\pm 1\}$ and
\begin{equation}\label{elem}
    \mathcal{E}= \begin{pmatrix}
    -1 & \sqrt{2}& 1 \\ -\sqrt{2} & 1 & 0\\ 1&0&0
\end{pmatrix};\ \ M_{z_j}= \begin{pmatrix}
    0&0&x_j \\ 0&e^{i\alpha_j}&0 \\ 1/x_j & 0&0
\end{pmatrix}
\end{equation}
for $z_j=x_je^{i\alpha_j}$.

To prove the next few results we shall define a \emph{positive expression} in the following way:
\begin{defn}\label{px}(Positive expression)
    An expression $P$ is termed a \emph{positive expression} if it is of the form 
    \[ P= \sum_{j\in J} c_j X_j e^{i\theta_j}\]
    where $J$ is a finite index set, $\{c_j\}_{j\in J} \subseteq \mathbb{R}^+$, 
    \[X_j=\frac{x_{i_1}x_{i_2}\cdots x_{i_q}}{x_{j_1}x_{j_2}\cdots x_{j_m}} \]
    for $\{x_{i_1}, x_{i_2},\cdots ,x_{i_q},x_{j_1},\cdots ,x_{j_m} \} \subseteq \mathcal{X}$ and \[ \theta_j=\sum_{j\in J_j}\alpha_j \] for all $\alpha_j \in \Theta$ and some finite index set $J_j$.
\end{defn}
\begin{rmk}
    Note that a ``positive expression" can still be a complex number. Here, we treat the invariants $x_j$ and $\alpha_j$ as formal symbols, allowing us to apply the triangle inequality for complex numbers later. So the term ``positive" in \emph{positive expression} refers exclusively to the positivity of the coefficients $c_j$ and independent of the values of $e^{i\theta_j}$. That is, we do not evaluate the values of $e^{i\theta_j}$ to determine whether an expression is positive. For instance, if $\pi \in \Theta$, then $e^{i\pi}$ is a positive expression, although $e^{i\pi}=-1$. 
\end{rmk}
\begin{example}
    The anti-diagonal elements of the matrix $M_{z_j}$ in \cref{elem} are all positive expressions.
\end{example}
\begin{notation}
    We denote a positive expression by $p_j$ for some $j$. Let $p_j(0)$ denote the expression $p_j$ with all $\alpha_j$ replaced by 0.
\end{notation}
The following proposition is straightforward, and we leave the proof to the reader.
\begin{prop}
    Let $p_j$ be a positive expression and $c > 0$. Then
    \begin{enumerate}
        \item $cp_j$ is also a positive expression.
        
        \item Finite sum of positive expressions is a positive expression. 

        \item Finite product of positive expressions is also a positive expression.
    \end{enumerate} 
\end{prop}
\begin{lem}\label{ineq}
    Let \[ p_t= \sum_{j\in J} c_j X_j e^{i\theta_j}\] be a positive expression for some $J,c_j,X_j$ and $\theta_j$ as in \Cref{px}. Let $p_t(0)$ denote the positive expression $p_t$ with all its $\theta_j$ replaces by $0$. Then \[ | p_t| \le \sum_{j\in J} c_j X_j = p_t(0). \]
\end{lem}
\begin{proof}
    It follows immediately from the triangle inequality of complex numbers, since $c_j$ and $X_j$ are all real and positive.
\end{proof}
\begin{prop}\label{mainprop}
    Let $\gamma \in \pi_1(S_{g,k})$. Then $\rho(\gamma) \in \puto$ has a matrix representative of the form 
    \begin{itemize}
        \item $\begin{pmatrix}
            p_1 & -p_2 & -p_3\\ -p_4 & p_5 & p_6 \\ -p_7 & p_8 & p_9
        \end{pmatrix}$ when $\gamma$ is non-peripheral, and \\
        \item  $\begin{pmatrix}
            p_1 & 0 & 0\\ -p_4 & p_5 & 0 \\ -p_7 & p_8 & p_9
        \end{pmatrix}$ or,   $\begin{pmatrix}
            p_1 & -p_2 & -p_3\\ 0 & p_5 & p_6 \\ 0 & 0 & p_9
        \end{pmatrix}$,  when $\gamma$ is peripheral
    \end{itemize}
    for some positive expressions $p_i$.
\end{prop}
\begin{proof}
    We have seen that the building blocks of $\rho(\gamma)$ (up to conjugacy)  are
    \begin{equation}\label{me}
        M_{z_1}\mathcal{E} = 
        \begin{pmatrix}
            x_1 & 0&0\\
            -\sqrt{2}e^{i\alpha_1} & e^{i\alpha_1} & 0\\
            -\frac{1}{x} & \frac{\sqrt{2}}{x} & \frac{1}{x}
        \end{pmatrix} = 
        \begin{pmatrix}
            p_1 & 0&0\\
            -p_2 & p_3 & 0 \\
            -p_4 & p_5 & p_6
        \end{pmatrix} 
    \end{equation} and
    \begin{equation}\label{mei}
        M_{z_2}\mathcal{E}^{-1}=
        \begin{pmatrix}
            x_2&-\sqrt{2}x_2 & -x_2\\
            0& e^{i\alpha_2} & \sqrt{2}e^{i\alpha_2} \\
            0&0&\frac{1}{x_2}
        \end{pmatrix}=
        \begin{pmatrix}
            p_7 & -p_8 & -p_9 \\
            0& p_{10} & p_{11}\\
            0&0& p_{12}
        \end{pmatrix}
    \end{equation} for some positive expressions $p_{j}$.\\
    \textbf{Case I:} When $\gamma$ is a non-peripheral loop: 
    We know that $\rho(\gamma)$ contains a block of the form $M_{z_1}\mathcal{E}M_{z_2}\mathcal{E}^{-1}$ or $M_{z_2}\mathcal{E}^{-1}M_{z_1}\mathcal{E}$ for some $z_1, z_2 \in \mathcal{Z}$. We see that
    \begin{equation}\label{memei}
    \begin{split}
        M_{z_1}\mathcal{E}M_{z_2}\mathcal{E}^{-1} &= 
    \begin{pmatrix}
        x_1x_2 & -\sqrt{2}x_1x_2 & -x_1x_2 \\
        -\sqrt{2}x_2e^{i\alpha_1} & 2x_2e^{i\alpha_1}+e^{i(\alpha_1+ \alpha_2)} & \sqrt{2}x_2e^{i\alpha_1}+\sqrt{2}e^{i(\alpha_1+ \alpha_2)} \\
        -\frac{x_2}{x_1} & \sqrt{2}\frac{x_2}{x_1}+\frac{\sqrt{2}}{x_1}e^{i\alpha_2} & \frac{x_2}{x_1}+\frac{2}{x_1}e^{i\alpha_2}+\frac{1}{x_1x_2}
    \end{pmatrix} \\
    &= \begin{pmatrix}
        p_{j_1} & -p_{j_2} & -p_{j_3} \\
        -p_{j_4} & p_{j_5} & p_{j_6} \\
        -p_{j_7} & p_{j_8} & p_{j_9}
    \end{pmatrix}
    \end{split}
    \end{equation} 
    for some positive expressions $p_{j_i}$, and 
    \begin{equation}\label{meime}
    \begin{split}
         M_{z_2}\mathcal{E}^{-1}M_{z_1}\mathcal{E} &=
         \begin{pmatrix}
             x_1x_2 + 2x_2e^{i\alpha_1}+ \frac{x_2}{x_1} & -\sqrt{2}x_2e^{i\alpha_1}-\sqrt{2}\frac{y}{x} & -\frac{x_2}{x_1}  \\ -\sqrt{2}e^{i(\alpha_1 +\alpha_2)}-\frac{\sqrt{2}}{x}e^{i\alpha_2} & e^{i(\alpha_1 +\alpha_2)}+\frac{2}{x}e^{i\alpha_2} & \frac{\sqrt{2}}{x}e^{i\alpha_2}\\
             -\frac{1}{x_1x_2}& \frac{\sqrt{2}}{x_1x_2} & \frac{1}{x_1x_2}
         \end{pmatrix} \\
         &= \begin{pmatrix}
             p_{k_1} & -p_{k_2} & -p_{k_3} \\
             -p_{k_4} & p_{k_5} & p_{k_6} \\
             -p_{k_7} & p_{k_8} & p_{k_9}
         \end{pmatrix}
    \end{split}
    \end{equation} for some positive expressions $p_{k_i}$. If $\rho(\gamma)=M_{z_1}\mathcal{E}M_{z_2}\mathcal{E}^{-1}$ or $\rho(\gamma)=M_{z_2}\mathcal{E}^{-1}M_{z_1}\mathcal{E}$, then the result is proved from the above two equations. If not, then we shall show that pre-multiplying and post-multiplying $M_{z_1}\mathcal{E}M_{z_2}\mathcal{E}^{-1}$ and $M_{z_2}\mathcal{E}^{-1}M_{z_1}\mathcal{E}$ with a building block of the form $M_{z_3}\mathcal{E}$ or $M_{z_3}\mathcal{E}^{-1}$ does not change the entries of the matrices in terms of being a positive expression. 
    
    Suppose $\rho(\gamma)$ contains a block of the form $M_{z_1}\mathcal{E}M_{z_2}\mathcal{E}^{-1}$. Then pre-multiplying it with a building block of the form $M_{z_3}\mathcal{E}$ yields 
    \begin{equation}
    \begin{split}
        M_{z_1}\mathcal{E}&M_{z_2}\mathcal{E}^{-1} \cdot M_{z_3}\mathcal{E} 
        =\begin{pmatrix}
             p_{k_1} & -p_{k_2} & -p_{k_3} \\
             -p_{k_4} & p_{k_5} & p_{k_6} \\
             -p_{k_7} & p_{k_8} & p_{k_9}
         \end{pmatrix} 
         \begin{pmatrix}
            p_{j_{10}} & 0&0\\
            -p_{j_{11}} & p_{j_{12}} & 0 \\
            -p_{j_{13}} & p_{j_{14}} & p_{j_{15}}
        \end{pmatrix} \\ 
        &= \begin{pmatrix}
            p_{k_1}p_{j_{10}}+ p_{k_2}p_{j_{11}} + p_{k_3}p_{j_{13}} & -p_{k_2}p_{j_{12}}- p_{k_3}p_{j_{14}} & -p_{k_3}p_{j_{15}} \\
            -p_{k_4}p_{j_{10}}- p_{k_5}p_{j_{11}} - p_{k_6}p_{j_{13}} & p_{k_5}p_{j_{12}}+ p_{k_6}p_{j_{14}} & p_{k_6}p_{j_{15}} \\
            -p_{k_7}p_{j_{10}}- p_{k_8}p_{j_{11}} - p_{k_9}p_{j_{13}} &  p_{k_8}p_{j_{12}}+p_{k_9}p_{j_{14}} & p_{k_9}p_{j_{15}}
        \end{pmatrix} \\
        &= \begin{pmatrix}
            p_{j_{16}} & -p_{j_{17}} & -p_{j_{18}} \\ -p_{j_{19}} & p_{j_{20}} & p_{j_{21}} \\ -p_{j_{22}} & p_{j_{23}} & p_{j_{24}}
        \end{pmatrix}
    \end{split}
    \end{equation}
    for some positive expression $p_{j_i}$, which is of same type as that of $M_{z_1}\mathcal{E}M_{z_2}\mathcal{E}^{-1}$ in terms of the entries being a positive expression (see \cref{memei}). Also, pre-multiplying $M_{z_1}\mathcal{E}M_{z_2}\mathcal{E}^{-1}$ with a building block of the form $M_{z_3}\mathcal{E}^{-1}$ yields 
    \begin{equation}
    \begin{split}
       M_{z_1}\mathcal{E}&M_{z_2}\mathcal{E}^{-1} \cdot M_{z_3}\mathcal{E}^{-1} 
            =\begin{pmatrix}
             p_{k_1} & -p_{k_2} & -p_{k_3} \\
             -p_{k_4} & p_{k_5} & p_{k_6} \\
             -p_{k_7} & p_{k_8} & p_{k_9}
         \end{pmatrix}
         \begin{pmatrix}
            p_{j_{25}} & -p_{j_{26}} & -p_{j_{27}} \\
            0& p_{j_{28}} & p_{j_{29}}\\
            0&0& p_{j_{30}}
        \end{pmatrix} \\
    &= \begin{pmatrix}
     p_{j_{25}} p_{k_{1}} & -p_{j_{26}} p_{k_{1}} - p_{j_{28}} p_{k_{2}} & -p_{j_{27}} p_{k_{1}} - p_{j_{29}} p_{k_{2}} - p_{j_{30}} p_{k_{3}} \\
    -p_{j_{25}} p_{k_{4}} & p_{j_{26}} p_{k_{4}} + p_{j_{28}} p_{k_{5}} & p_{j_{27}} p_{k_{4}} + p_{j_{29}} p_{k_{5}} + p_{j_{30}} p_{k_{6}} \\
    -p_{j_{25}} p_{k_{7}} & p_{j_{26}} p_{k_{7}} + p_{j_{28}} p_{k_{8}} & p_{j_{27}} p_{k_{7}} + p_{j_{29}} p_{k_{8}} + p_{j_{30}} p_{k_{9}}
    \end{pmatrix} \\
    &= \begin{pmatrix}
            p_{j_{31}} & -p_{j_{32}} & -p_{j_{33}} \\ -p_{j_{34}} & p_{j_{35}} & p_{j_{36}} \\ -p_{j_{37}} & p_{j_{38}} & p_{j_{39}}
        \end{pmatrix}
    \end{split}
\end{equation} for some positive expression $p_{j_i}$, which is also of same type as that of $M_{z_1}\mathcal{E}M_{z_2}\mathcal{E}^{-1}$ in terms of the entries being a positive expression (see \cref{memei}). Similarly, post-multiplying $M_{z_1}\mathcal{E}M_{z_2}\mathcal{E}^{-1}$ with the building blocks, we get 
    \begin{equation}
    \begin{split}
        M_{z_3}\mathcal{E} \cdot M_{z_1}\mathcal{E}M_{z_2}\mathcal{E}^{-1} &= \begin{pmatrix}
            p_{j_{10}} & 0&0\\
            -p_{j_{11}} & p_{j_{12}} & 0 \\
            -p_{j_{13}} & p_{j_{14}} & p_{j_{15}}
        \end{pmatrix} \begin{pmatrix}
             p_{k_1} & -p_{k_2} & -p_{k_3} \\
             -p_{k_4} & p_{k_5} & p_{k_6} \\
             -p_{k_7} & p_{k_8} & p_{k_9}
         \end{pmatrix} \\
         &= \begin{pmatrix}
            p_{j_{40}} & -p_{j_{41}} & -p_{j_{42}} \\ -p_{j_{43}} & p_{j_{44}} & p_{j_{45}} \\ -p_{j_{46}} & p_{j_{47}} & p_{j_{48}}
        \end{pmatrix}
    \end{split}
    \end{equation} and 
\begin{equation}
 \begin{split}
    M_{z_3}\mathcal{E}^{-1} \cdot M_{z_1}\mathcal{E}M_{z_2}\mathcal{E}^{-1} &= \begin{pmatrix}
            p_{j_{25}} & -p_{j_{26}} & -p_{j_{27}} \\
            0& p_{j_{28}} & p_{j_{29}}\\
            0&0& p_{j_{30}}
        \end{pmatrix} \begin{pmatrix}
             p_{k_1} & -p_{k_2} & -p_{k_3} \\
             -p_{k_4} & p_{k_5} & p_{k_6} \\
             -p_{k_7} & p_{k_8} & p_{k_9}
         \end{pmatrix} \\
         &= \begin{pmatrix}
            p_{j_{49}} & -p_{j_{50}} & -p_{j_{51}} \\ -p_{j_{52}} & p_{j_{53}} & p_{j_{54}} \\ -p_{j_{55}} & p_{j_{56}} & p_{j_{57}}
        \end{pmatrix} 
\end{split}
\end{equation} for some positive expression $p_{j_i}$, which are again of same type as that of $M_{z_1}\mathcal{E}M_{z_2}\mathcal{E}^{-1}$ in terms of the entries being a positive expression (see \cref{memei}). This concludes the result when $\rho(\gamma)$ contains a block of the form $M_{z_1}\mathcal{E}M_{z_2}\mathcal{E}^{-1}$. Now notice that, the entries of $M_{z_2}\mathcal{E}^{-1}M_{z_1}\mathcal{E}$ and $M_{z_1}\mathcal{E}M_{z_2}\mathcal{E}^{-1}$ are of same form in terms of being positive expressions (see \cref{memei} and \cref{meime}). So, it concludes the result when $\rho(\gamma)$ contains a block of the form $M_{z_2}\mathcal{E}^{-1}M_{z_1}\mathcal{E}$ also.\\
\textbf{Case II:} When $\gamma$ is a peripheral loop:\\
We know that 
\[ \rho(\gamma)= M_{z_r}\mathcal{E}^{\delta}\cdots M_{z_1}\mathcal{E}^{\delta} \] for some $z_i \in \mathcal{Z}$ and $\delta \in \{\pm 1 \}$, i.e., powers of all the $\mathcal{E}$ are same (see \Cref{peripheral} and \cite[Section 5.3]{CTT}). When $\delta=1$, from \cref{me}, we get
\begin{equation}\label{peri1}
    \rho(\gamma) =
\begin{pmatrix}
    x_1x_2\cdots x_r & 0 &0\\
    * & e^{i(\alpha_1+\cdots +\alpha_r)} & 0\\
    *&*& \frac{1}{x_1x_2\cdots x_r}
\end{pmatrix}
\end{equation}
and for $\delta=-1$, from \cref{mei}, we get
\begin{equation}\label{peri2}
    \rho(\gamma) =
\begin{pmatrix}
    x_1x_2\cdots x_r & * &*\\
    0 & e^{i(\alpha_1+\cdots +\alpha_r)} & *\\
    0&0& \frac{1}{x_1x_2\cdots x_r}
\end{pmatrix}.
\end{equation}
In both the cases, we see that the diagonal entries are positive expressions.
\end{proof}
\begin{figure}
\centering
\begin{tikzpicture}
  \draw (0,0)--(135:5)--(105:5)--(0,0)--(75:5)--(105:5) (0,0)--(45:5)--(75:5); 
  \draw [line width=0.5mm] (141.43:3.4)--(126.43:3.4)--(113.57:3.4)--(96.43:3.4)--(83.57:3.4)--(66.43:3.4)--(53.57:3.4)--(38.57:3.4);
  \draw [line width=0.5mm] (126.43:3.4)--(120:4.2)--(113.57:3.4) (96.43:3.4)--(0,4.2)--(83.57:3.4) (66.43:3.4)--(60:4.2)--(53.57:3.4);
   \draw [line width=0.5mm] (120:4.2)--(120:5.2) (0,4.2)--(0,5.2) (60:4.2)--(60:5.2);
   \draw (-0.01,0.14) node[anchor=north west] {$\widetilde{p}$} (121:3.38) node[rotate=30] {$>$} (0,3.38) node {$>$} (60:3.38) node[rotate=-30] {$>$} (35:3.42) node[rotate=-45] {$\cdots$} (60:5.4) node[rotate=60] {$\cdots$} (0,5.4) node[rotate=90] {$\cdots$} (120:5.4) node[rotate=120] {$\cdots$} (142:3.4) node[rotate=45] {$\cdots$} (0,2.5) node[red] {$\widetilde{\gamma}$} (1.4,1) node {$T$} (122.9:3.8) node[rotate=-80] {$>$} (117.1:3.8) node[rotate=145] {$>$} (92.9:3.8) node[rotate=-120] {$>$} (87:3.8) node[rotate=120] {$>$} (62.9:3.8) node[rotate=215] {$>$} (57:3.8) node[rotate=85] {$>$};
   \draw[->, thick] (1,4.4) arc[start angle=0, end angle=270, radius=2mm];
   \draw[->, thick] (3,3.5) arc[start angle=0, end angle=270, radius=2mm];
    \draw[->, thick] (-1.3,4.2) arc[start angle=0, end angle=270, radius=2mm];
   \draw[line width=0.5mm,  red] (2.5,1.8) arc[start angle=45, end angle=135, radius=3.5cm];
   \node [red, rotate=-45] at (34:3.08) {$\cdots$};
   \node [red, rotate=45] at (145:3.03) {$\cdots$};
\end{tikzpicture}
\caption{Segment of a lift of a peripheral loop $\gamma$.}
\label{peripheral}
\end{figure}
\begin{cor}\label{trpos}
    For $\gamma \in \pi_1(S_{g,k}),\ \rho(\gamma)$ has a representative in $\mathrm{U}(2,1)$ such that its trace $\mathrm{tr}(\rho(\gamma))$ is a positive expression.
\end{cor}
\begin{notation}
    For two matrices $A = (a_{ij})_{n \times n}$ and $B = (b_{ij})_{n \times n}$, we denote $|A| = (|a_{ij}|)_{n \times n}$, and we write $A \le B$ if $a_{ij} \le b_{ij}$ for all $i, j$.
\end{notation}
To prove our result, we will use the following results from matrix analysis:
\begin{thm}\cite[Theorem 8.1.18]{HJ}\label{spectralIneq}
    Let $A=(a_{ij}),B=(b_{ij}) \in M_n$ such that $|A| \le B$. Then $\sigma(A) \le \sigma (|A|) \le \sigma(B)$.
\end{thm}
\begin{lem}\label{samesigma}
    Let \[ A=\begin{pmatrix}
        a & b & c\\ d & f & h \\j & l & m
    \end{pmatrix} \text{ and } \ B=\begin{pmatrix}
        a & -b & -c\\ -d & f & h \\-j & l & m
    \end{pmatrix}. \] Then $A$ and $B$ have same set of eigenvalues.
\end{lem}
\begin{proof}
    It is enough to show that the characteristic polynomials of $A$ and $B$ are same. We see that \begin{equation*}
    \begin{split}
       det(B-xI)= \begin{vmatrix}
        a-x & -b & -c\\ -d & f-x & h \\ -j & l & m-x
    \end{vmatrix} &= -\begin{vmatrix}
        -(a-x) & b & c\\ -d & f-x & h \\ -j & l & m-x
    \end{vmatrix} \\ &= +\begin{vmatrix}
        a-x & b & c\\ d & f-x & h \\ j & l & m-x
    \end{vmatrix} =det(A-xI).
    \end{split}   
    \end{equation*}

We multiply the first row and first column with $-1$ to get the second and third equality respectively. This proves the lemma.
\end{proof}
\begin{rmk}
    Note that $A$ and $B$ are conjugate by the matrix $S=\begin{pmatrix}
    1&0&0\\0&-1&0\\0&0&-1
    \end{pmatrix}.$
\end{rmk}
\noindent \emph{Proof of \Cref{mainthm1}}.
For a given $T$-bent representation $\rho:\pi_1(S_{g,k}) \rightarrow \puto$, we determine the framing $\phi :\mathcal{F}_{\infty} \rightarrow \partial \HtC$ and get the edge invariants $\{z_1,z_2,\cdots,z_{3(2g-2+k)}\}$ associated with some fixed ideal triangulation $T$. Then we replace each of the invariants with their absolute values $\{x_1,x_2,\cdots,x_{3(2g-2+k)}\}$ and define $\rho_0:\pi_1(S_{g,k}) \rightarrow \poto$ to be the representation corresponding to these real and positive edge-invariants. Now from \cite[Theorem 2]{WillBending}, we know that $\rho_0$ is discrete and faithful.

We claim that $\rho_0$ dominates $\rho$ in the Bergman translation length spectrum. To prove it, let $\gamma \in \pi_1(S_{g,k})$ be a non-peripheral curve. Then from \Cref{mainprop}, we get, \[ \rho(\gamma) = \begin{pmatrix}
            p_1 & -p_2 & -p_3\\ -p_4 & p_5 & p_6 \\ -p_7 & p_8 & p_9
        \end{pmatrix} \]
for some positive expressions $p_i$. Then clearly,  \[ \rho_0(\gamma) = \begin{pmatrix}
            p_1(0) & -p_2(0) & -p_3(0)\\ -p_4(0) & p_5(0) & p_6(0) \\ -p_7(0) & p_8(0) & p_9(0)
        \end{pmatrix}. \] Let $B(\gamma)$ denote the matrix \[B(\gamma):=\begin{pmatrix}
            p_1(0) & p_2(0) & p_3(0)\\ p_4(0) & p_5(0) & p_6(0) \\ p_7(0) & p_8(0) & p_9(0)
        \end{pmatrix}.\]
    Then using \Cref{ineq}, we get \[ |\rho(\gamma)| = \begin{pmatrix}
            |p_1| & |-p_2| & |-p_3| \\ |-p_4| & |p_5| & |p_6| \\ |-p_7| & |p_8| & |p_9|
        \end{pmatrix} \le \begin{pmatrix}
            p_1(0) & p_2(0) & p_3(0)\\ p_4(0) & p_5(0) & p_6(0) \\ p_7(0) & p_8(0) & p_9(0)
        \end{pmatrix} =B(\gamma).\]
    Then from \Cref{spectralIneq}, we get
    \begin{equation}\label{spectral_sb}
         \sigma(\rho(\gamma)) \le \sigma(B(\gamma)). 
    \end{equation} Now applying \Cref{samesigma} on $B(\gamma)$ and $\rho_0(\gamma)$, we have \[ \sigma(B(\gamma)) = \sigma(\rho_0(\gamma)). \] From the above two equations, we get \[ \sigma(\rho(\gamma)) \le \sigma(\rho_0(\gamma)) \implies \ell(\rho(\gamma)) = 2\ln \sigma(\rho(\gamma)) \le 2\ln \sigma(\rho_0(\gamma))=\ell(\rho_0(\gamma)). \]

    If $\gamma$ is a peripheral loop, then from \cref{peri1} and \cref{peri2}, we know that 
\[  \rho(\gamma) =
\begin{pmatrix}
    x_1x_2\cdots x_r & 0 &0\\
    * & e^{i(\alpha_1+\cdots +\alpha_r)} & 0\\
    *&*& \frac{1}{x_1x_2\cdots x_r}
\end{pmatrix} \text{ or,  }
\begin{pmatrix}
    x_1x_2\cdots x_r & * &*\\
    0 & e^{i(\alpha_1+\cdots +\alpha_r)} & *\\
    0&0& \frac{1}{x_1x_2\cdots x_r}
\end{pmatrix}
\] and \[  \rho_0(\gamma) =
\begin{pmatrix}
    x_1x_2\cdots x_r & 0 &0\\
    * & 1 & 0\\
    *&*& \frac{1}{x_1x_2\cdots x_r}
\end{pmatrix} \text{ or,  }
\begin{pmatrix}
    x_1x_2\cdots x_r & * &*\\
    0 & 1 & *\\
    0&0& \frac{1}{x_1x_2\cdots x_r}
\end{pmatrix}\ \text{ respectively}.
\] If $x_1x_2\cdots x_r \neq 1$, then $\rho(\gamma)$ and $\rho_0(\gamma)$ are loxodromic. Then \[ \ell (\rho(\gamma))= 2 |\ln (x_1x_2\cdots x_r)|=\ell(\rho_0(\gamma)). \] If  $x_1x_2\cdots x_r = 1$, then $\ell(\rho(\gamma))=0=\ell(\rho_0(\gamma))$. This completes the proof.
\qed

\begin{cor}\label{trace_ineq2}
Let $\rho$ and $\rho_0$ be as in \Cref{mainthm1}. Then \[ |\mathrm{tr}(\rho(\gamma))| \le |\mathrm{tr}(\rho_0(\gamma))| \] for all $\gamma \in \pi_1(S_{g,k})$.
\end{cor}
\begin{proof}
    To prove the corollary, we shall use the fact that the smooth function $f:\mathbb{R}^+ \rightarrow \mathbb{R}^+$ defined by \[ f(x)=x+\frac{1}{x} \] is monotone increasing on $[1,\infty)$ and attains its global minima at $x=1$ with $f(1)=2$.

    Let $\{re^{i\phi},r^{-1}e^{i\phi},e^{-2i\phi}\}$ and $\{r_0,r_0^{-1},1\}$ be the eigenvalues of $\rho(\gamma)$ and $\rho_0(\gamma)$ respectively. Then \[|\mathrm{tr}(\rho(\gamma))| = |re^{i\phi}+r^{-1}e^{i\phi}+e^{-2i\phi}| \le r+r^{-1}+1 \le r_0+r_0^{-1}+1 =|\mathrm{tr}(\rho_0(\gamma))|. \]
    This completes the proof. 
\end{proof}
\begin{rmk}
    Note that, \Cref{trace_ineq} can be independently proved from \Cref{trpos}. Let $\mathrm{tr}(\rho(\gamma))=p_t$ for some positive expression $p_t$. Then we can see that $ \mathrm{tr}(\rho_0(\gamma)) = p_t(0)$. Now from \Cref{ineq}, we get \[ |\mathrm{tr}(\rho(\gamma))| = |p_t| \le p_t(0) = |\mathrm{tr}(\rho_0(\gamma))|. \]
\end{rmk}

\appendix
\section*{Appendix}
\label{sec:appendix} While domination holds with respect to the Bergman translation length and the modulus of the trace, it is a natural question whether the domination is preserved under the discriminant function  \[ f(z) = |z|^4 - 8\Re(z^3) + 18|z|^2 - 27 \]  which is evaluated at the traces of $\rho(\gamma)$ and $\rho_0(\gamma)$ (see \Cref{discri}). We thank John R. Parker for asking this question. In this Appendix we collect evidences to see that the function $f$ exhibits more subtle behavior and does not satisfy a similar domination property. To better understand this discrepancy, we conducted numerical experiments using \emph{Mathematica}. These computations involve evaluating $f(\mathrm{tr}(\rho(\gamma)))$ and $f(\mathrm{tr}(\rho_0(\gamma)))$ where $\rho(\gamma)=M_{z_1}\mathcal{E}M_{z_2}\mathcal{E}^{-1}$ but with randomly sampled parameters $z_1,z_2$ in order to test whether domination holds for the discriminator function.

The following code samples random parameters $x, y \in (0,5)$ and $ a, b \in \left(0, \frac{\pi}{2}\right)$ where $z_1=xe^{ia}$ and $z_2=ye^{ib}$, constructs $\rho(\gamma)$ and $\rho_0(\gamma)$, and evaluates the discriminator function at their traces. Whenever the inequality
\[
f(\mathrm{tr}(\rho(\gamma))) > f(\mathrm{tr}(\rho_0(\gamma)))
\]
holds, indicating a failure of domination, the code outputs the sampled parameters along with the corresponding function values. Otherwise, it prints a confirmation that domination is preserved.

\subsection*{Mathematica code}
\begin{verbatim}
f[x_] := Abs[x]^4 - 8 Re[x^3] + 18 Abs[x]^2 - 27;
T = {{-1, Sqrt[2], 1}, {-Sqrt[2], 1, 0}, {1, 0, 0}};
    (*The matrix \[Epsilon] in equation 18, Will*)

M[x_, a_] := {{0, 0, x}, {0, Exp[I a], 0}, {1/x, 0, 0}};
    (*The matrix M_{x,\[Alpha]} in equation 13, Will*)

A[x1_, x2_, a1_, a2_] := M[x1, a1].T.M[x2, a2].Inverse[T];
For[j = 1, j <= 10, j++,
 Module[{x, y, a, b, A1, A2, trace, TRACE, fIneq},
  {x, y} = RandomReal[5, 2];
  {a, b} = RandomReal[{0, Pi/2}, 2];
  A1 = A[x, y, a, b];
  A2 = A[x, y, 0, 0];
  
  trace = Tr[A1];(* trace of A1 *)
  TRACE = Tr[A2];(* 
  trace of A2 *)
  
  fIneq = 
   If[f[trace] > f[TRACE], 
    Row[{"{x,y,a,b} = {", x, ", ", y, ", ", a, ", ", b, 
      "};  f(trace) = ", f[trace], ", f(TRACE) = ", f[TRACE]}], 
    "True!"];
  
  Print[Row[{fIneq}]];
  ]]
\end{verbatim}
\subsection*{Sample Output}
\begin{verbatim}
{x,y,a,b} = {3.0497, 2.0373, 0.2936, 0.0886}; 
f(trace) = 13328., f(TRACE) = 12855.2
True!
True!
{x,y,a,b} = {1.8096, 2.0235, 1.3414, 0.1429}; 
f(trace) = 8837.83, f(TRACE) = 6742.79
True!
True!
True!
True!
True!
True!
\end{verbatim}

\bibliographystyle{alpha}
\bibliography{references}
\end{document}